\documentclass[12pt]{amsart}

\usepackage[english]{babel}

\usepackage{amsthm}

\usepackage{cite}

\usepackage{enumerate}
\usepackage{amsfonts,amsmath,amsthm,amscd,amssymb,latexsym}

\begin{document}

\newtheorem{theorem}{Theorem}
\newtheorem{lemma}{Lemma}
\newtheorem{proposition}{Proposition}
\newtheorem{corollary}{Corollary}
\newtheorem{definition}{Definition}
\theoremstyle{definition}
\newtheorem{example}{Example}
\newtheorem{remark}{Remark}
\newenvironment{romanlist}{
	\def\theenumi{\roman{enumi}}\def\theenumii{\alph{enumii}}
	\def\labelenumi{(\theenumi)}\def\labelenumii{(\theenumii)}%
	\let\item\Item
	\begin{enumerate}
	}{
	\end{enumerate}}
	
	\let\Item\item
\newenvironment{enumeroman}{%
  \def\theenumi{\roman{enumi}}\def\theenumii{\alph{enumii}}%
  \def\labelenumi{(\theenumi)}\def\labelenumii{(\theenumii)}%
		\let\item\Item
  \begin{enumerate}%
}{%
  \end{enumerate}}
	
\def\address#1{\expandafter\def\expandafter\@aabuffer\expandafter
	{\@aabuffer{\affiliationfont{#1}}\relax\par
	\vspace*{13pt}}}

\setcounter{page}{1}

\title{Primitive irreducible representations of finitely generated
nilpotent groups}

\author{Anatolii.~V.~Tushev}

\address{Department of Mathematics, Dnipro National University,\\
Gagarin Avenue 72,Dnipro, 49065, Ukraine\\
anavlatus@gmail.com}

\maketitle

\begin{center}
 
\small{Department of Mathematics, Dnipro National University,\\
Gagarin Avenue 72, Dnipro, 49065, Ukraine\\
anavlatus@gmail.com}

\end{center}

\begin{abstract}
In the  paper we show that any irreducible representation of a
finitely generated nilpotent group $G$ over a finitely generated field $F$
of characteristic zero is induced from a primitive representation of some
subgroup of $G$.
\end{abstract}

\section{Introduction}

Let $G$ be a group and let $R$ be a ring. Let $H$ be a subgroup of the
group $G$ and let $U$ be a right $RH$-module. Since the group ring $RG$ can
be considered as a left $RH$-module, we can define the tensor product 
$U\otimes _{RH}RG$, which is a right $RG$-module named the $RG$-module
induced from the $RH$-module $U$. If $M$ is an $RG$-module and $U\leq M$,
then it follows from \cite[Chap. 2, Lemma 1.1]{Karp90} that 
$M=U\otimes _{RH}RG$  if and only if $M=\oplus_{t\in T}Ut$, where $T$ 
is a right transversal to the subgroup $H$ in $G$.\par

If a group $G$ has a finite series each of whose factor is either cyclic or
quasi-cyclic then $G$ is said to be minimax. If in such a series all
infinite factors are cyclic then the group $G$ is said to be polycyclic; the
number $h(G)$ of infinite factors in such a series is the Hirsch number of $G$. 
By \cite[Theorem 2.13]{Wehr09}, any finitely generated nilpotent group 
is polycyclic. \par 

Let $G$ be a group, let $F$ be a field, let $\varphi$ be a representation
of $G$ over $F$ and let $M$ be an $FG$-module of the representation $\varphi$. 
The representation  $\varphi$  is said to be faithful if $Ker \varphi =1$. 
If $M$ is induced from some $FH$-module $U$, where $H$ is a
subgroup of the group $G$, then we say that the representation $\varphi $ is
induced from a representation $\phi $ of the subgroup $H$, where $U$ is the
module of the representation $\phi $. The module $M$ and the representation 
$\varphi$ are said to be primitive if there are no subgroups $H<G$ such that  
$M$ is  induced from an $FH$-submodule. If the group $G$ is polycyclic then the
module $M$ and the representation $\varphi$ are said to be semi-primitive if 
there are no subgroups $H\leq G$ such that $h(H)<h(G)$ and $M$ is  induced from 
an $FH$-submodule.\par

There are many results which show that the existence of a faithful
irreducible representation of a group $G$ over a field $F$ may have
essential influence on the structure of the group $G$ (see for instance \cite
{Szec16,SzTu17,Tush90,Tush93,Tush2000,Tush02,Tush12}). Certainly, primitive
irreducible modules are a basic subject for investigations when we are
dealing with irreducible representations.\par

In \cite{Harp77} Harper solved a problem raised by Zaleskii and proved that
any not abelian-by-finite finitely generated nilpotent group has an
irreducible primitive representation over a not locally finite field. In 
\cite{Tush02} we proved that if a minimax nilpotent group $G$ of class 2 has
a faithful irreducible primitive representation over a finitely generated
field of characteristic zero then the group $G$ is finitely generated. In 
\cite{Harp80} Harper studied polycyclic groups which have faithful
irreducible semi-primitive representations. It is well known that any
polycyclic group is finitely generated soluble of finite rank and meets the
maximal condition for subgroups (in particular, for normal subgroups). In 
\cite{Tush2000} we showed that in the class of soluble groups of finite rank
with the maximal condition for normal subgroups only polycyclic groups may
have faithful irreducible primitive representations over a field of
characteristic zero.\par

It is well known that any finitely dimensional irreducible representation 
of a group $G$  is induced from a primitive representation of some 
subgroup of $G$. It is evident that any representation of a polycyclic group 
$G$ over a field $F$ is induced from a semi-primitive representation of some 
subgroup $H$ of $G$.  We say that a semi-primitive representation of $G$ is 
purely semi-primitive if it is not induced from a primitive representation of some 
subgroup $H$ of $G$ and it is still unknown whether 
polycyclic groups may have purely semi-primitive representations. 
In the paper we show that under some additional conditions on the group 
$G$ and the field $F$ semi-primitive irreducible representations 
do not occur.  Our main result (Theorem 1) states that any irreducible 
representation of a finitely generated nilpotent group $G$ over a finitely 
generated field $F$ of characteristic zero is induced from a primitive 
representation of some subgroup of $G$. So, the following questions 
naturally arise.\par 

\medskip

\textbf{Question 1} \textit{Is it true that any irreducible representation
of a finitely generated nilpotent group $G$ over a field $F$ is induced 
from a primitive representation of some subgroup of $G$?}\par

\medskip

\textbf{Question 2} \textit{Is it true that any irreducible representation
of a minimax nilpotent group $G$ over a field $F$  is induced from a 
primitive representation of some subgroup of $G$?}\par

\medskip

Our methods of investigations are based on the following approach introduced by Brookes in \cite{Broo88}. 
Let $N$ be a group and let $K$ be a normal subgroup of $N$ such that the quotient group $N/K$ is finitely 
generated free abelian. Let $R$ be a ring and let $W$ be a finitely 
generated $RN$-module. Let $I$ be an $N$-invariant ideal of $RK$ such 
that $\left| {K/{I^\dag }} \right| < \infty $ and $k = R/(R \cap I)$ is a field, 
where ${I^\dag } = K \cap (1 + I)$. Then $N/{I^\dag }$ has a central subgroup 
$A$ of finite index. So, the quotient module $W/WI$ may be considered as a 
finitely generated $kA$-module and we can apply powerful techniques of 
commutative algebra for studying properties of $W/WI$. The results of the 
paper show that properties of $W$ and $W/WI$ are deeply related. 

The main results of the paper were announced in  \cite{Tush21}. 
The author is deeply grateful to the referee for helpful notes.


\section{ Auxiliary results}

 
\begin{lemma}
Let $G$ be a group and let $D$ be a finite central subgroup of $G$. 
Let $\{H_{i}|\in \mathbb{N}\}$ be a descending chain of subgroups of $G$. 
Then:\par 

(i) if there is $n\in \mathbb{N}$ such that $D{H_{n}}=D{H_{i}}$ for all $
i\geq n$ then there is $m\in \mathbb{N}$ such that ${H_{m}}={H_{i}}$ for all 
$i\geq m$;\par 

(ii) if $\cap_{i\in \mathbb{N}}{H_{i}}=1$ then $\cap_{i\in\mathbb{N}}DH_i=D$.
\end{lemma}
\begin{proof}

(i) There is no harm in assuming that $n=1$. Since $\{D\cap {H_{i}}
|i\in\mathbb{N}\}$ is a descending chain of subgroups of $D$ and $|D|<\infty$, 
we can conclude that there is $m\in\mathbb{N}$ such that $D\cap {H_{m}} 
=D\cap {H_{i}}$ for all $i\geq m$. We also may assume that $m=1$. Thus, we
can assume that $K=D\cap {H_{1}}=D\cap {H_{i}}$ for all $i\in \mathbb{N}$. 
As the subgroup $D$ is central, passing to the quotient group $G/K$ we can 
assume that $1=D\cap {H_{1}}=D\cap {H_{i}}$ for all $i\in \mathbb{N}$. Then 
$DH_i=D\times {H_i}$ for all $i\in\mathbb{N}$ and, as $H_{i}\leq {H_{1}}$, 
we can conclude that $DH_i=DH_1$ if and only if ${H_{i}=H_{1}}$.\par  

(ii). Since $\cap_{i\in\mathbb{N}}{H_{i}}=1$, there is $m\in\mathbb{N}$ 
such that $D\cap {H_{i}}=1$ for all $i\geq m$. Evidently, there is no harm
in assuming that $m=1$. Then any element of $D{H_1}$ can be unequally
written in the form $dh$, where $d\in D$, $h\in {H_{1}}$, and hence $dh\in
D {H_{i}}$ if and only if $h\in {H_{i}}$. As $\cap_{i\in \mathbb{N}}{H_{i}}
=1$, it implies that $\cap_{i\in\mathbb{N}}D{H_{i}}=D$.
\end{proof}


\begin{lemma}
Let $G$ be a finitely generated nilpotent group and let $D$ be
the derived subgroup of $G$. Let $\{ {H_i} | i \in\mathbb{N} \}$ be a
descending chain of subgroups of finite index in $G$. Suppose that there is 
$n\in\mathbb{N}$ such that $DH_n = DH_i$ for all $i \geq n$. Then there
is $m\in\mathbb{N}$ such that $H_m = H_i$ for all $i\geq m$.
\end{lemma}

\begin{proof}
Proof is by induction on the length of the lower central series $1=
G_{0}\leq {G_{1}}\leq ...\leq {G_{k}}=G$ of the group $G$. Put $K={G_{1}}$
then by the induction hypothesis we can assume that there is $n\in\mathbb{N}$ 
such that $K{H_{n}}=K{H_{i}}$ for all $i\geq n$. Evidently, we can assume 
that $n=1$.\par  

It is well known that $[ab,c]={[a,c]^{b}}[b,c]$ and $[c,ab]=[c,b]{[c,a]^{b}}$ 
for any elements $a,b,c\in G$ (see \cite[Section 3.2]{KaMe79}). As $[{G_{2}} 
,G]={G_{1}}=K\leq Z(G)$, we see that if some of elements $a,c$ belongs to 
${G_{2}}$ then $[ab,c]=[a,c][b,c]$ and $[c,ab]=[c,b][c,a]$. It easily 
implies that if some of elements $a,b$ belongs to ${G_{2}}$ then

\begin{equation}
[{a^{u}},{b^{v}}]=[a,b]^{uv}  \label{1}
\end{equation}

for any $u,v\in \mathbb{N}$. Since $|G:{H_{1}}|<\infty $, there are $u,v\in 
\mathbb{N}$ such that ${a^{u}},{b^{v}}\in H_{1}$ for any $a\in {G_{2}}$ and
any $b\in G$. Then it follows from (\ref{1}) that $[a,b]^{uv}\in \lbrack
H_{1},H_{1}]$. Since the subgroup $K={G_{1}}=[G_{2},G]$ is abelian and
generated by commutators $[a,b]$, where $a\in G_{2}$ and $b\in G$, it
implies that $K^{uv}\leq \lbrack H_{1},H_{1}]\leq H_{1}$. As $KH_{1}=KH_{i}$
for all $i\geq n$, and $K$ is a central subgroup of $G$, we can conclude
that $K^{uv}\leq \lbrack {H_{1}},{H_{1}}]=[K{H_{1}},K{H_{1}}]=[K{H_{i}},K{
H_{i}}]=[{H_{i}},{H_{i}}]\leq {H_{i}}$ for all $i\in \mathbb{N}$.\par 

Since the subgroup $K^{uv}$ is central, we can pass to the quotient group $
G/K^{uv}$ and hence, as $|{K/}K^{uv}|<\infty$, we can assume that the
subgroup $K$ is finite. Then the assertion follows from Lemma 1.1(i).
\end{proof}

Let $p$ be a prime integer then $n\in \mathbb{N}$ is said to be a $p^{\prime
}$-number if $p$ does not divide $n$. A finite group is called
a $p^{\prime}$-group if its order is a $p^{\prime}$-number.

  
\begin{proposition}
Let $G$ be a finitely generated nilpotent group and let $D$
be the derived subgroup of $G$. Let $\{{H_{i}}|i\in\mathbb{N}\}$ be a
descending chain of subgroups of finite index in $G$ such that $|G:\cap
_{i\in \mathbb{N}}{H_{i}}|=\infty $. Then there is a descending chain 
$\{G_i|i\in\mathbb{N}\}$ of subgroups of finite index in $G$ such that 
$DH_i\leq G_i$, the quotient group $G/(\cap_{i\in\mathbb{N}}G_i)$
is free abelian  and:\par

(i) either for any prime integer $p$ there is $n\in\mathbb{N} $ such that $p
\notin \pi ({G_n}/{G_i})$ for all $i \geq n$;\par

(ii) or $G/{G_i}$ is a $p$-group for all $i\in\mathbb{N} $, where $p$ is a
prime integer.
\end{proposition}

\begin{proof}

It is easy to note that $|G:\cap_{i\in\mathbb{N}}{H_{i}}|<\infty $
if and only if there is $m\in\mathbb{N}$ such that ${H_{m}}={H_{i}}$ for all 
$i\geq m$. Then it follows from Lemma   2, that $|G/(\cap_{i\in\mathbb{N}}
D {H_{i})}|=\infty $ and hence there is no harm in assuming that $D\leq {
H_{i}} $ for all $i\geq m$. Then passing to the quotient group $G/D$ we can
assume that the group $G$ is abelian.\par  

(i) If for any prime integer $p$ there is $n\in \mathbb{N}$ such that $
p\notin \pi ({H_{n}}/{H_{i}})$ for all $i\geq n$ then we can put ${G_{i}}={
H_{i}}$ for all $i\in \mathbb{N}$.\par  

(ii) Suppose that the situation (i) does not hold for $\{{H_{i}}|i\in 
\mathbb{N}\}$. Then there is a prime integer $p$ such that for any $i\in 
\mathbb{N}$ there is $j\in \mathbb{N}$ such that $j>i$ and $p\in \pi ({H_{i}}
/{H_{j}})$. Let ${G_{i}}$ be the set of all elements $g\in G$ such that ${
g^{n}}\in {G_{i}}$ for some $p^{\prime }$-number $n\in \mathbb{N}$. It is
easy to check that ${G_{i}}$ is a subgroup of $G$ such that ${H_{i}}\leq {
G_{i}}$, ${G}/{H_{i}}=({G_{i}}/{H_{i}})\times ({S_{i}}/{H_{i}})$, where ${
G_{i}}/{H_{i}}$ is a $p^{\prime }$-group and ${S_{i}}/{H_{i}}$ is a $p$
-group. It implies that $G/{G_{i}}$ is a $p$-group for all $i\in \mathbb{N}$. 
As ${H_{i}}\geq H_{i+1}$, we can conclude that ${G_{i}}\geq G_{i+1}$ and 
hence $\{{G_{i}}|i\in \mathbb{N}\}$ is a descending chain of subgroups of finite
index in $G$. Since for any $i\in \mathbb{N}$ there is $j\in \mathbb{N}$
such that $p\in \pi ({H_{i}}/{H_{j}})$, we can conclude that for any $i\in 
\mathbb{N}$ there is $j\in \mathbb{N}$ such that $j>i$ and $|G/{G_{j}}|>|G/{
G_{i}}|$. It implies that $|G/(\cap_{i\in \mathbb{N}}{G_{i})}|=\infty $.
Passing to the quotient group $G/(\cap_{i\in \mathbb{N}}{G_{i})}$, we can
assume that $\cap_{i\in \mathbb{N}}{G_{i}}=1$. As the group $G$ is finitely 
generated abelian, the torsion subgroup $T$ of $G$ is finite. Then it follows 
from Lemma   1(ii) that $\cap_{i\in \mathbb{N}}{TG_{i}}=T$. Therefore, 
changing $G_{i}$ by $TG_{i}$ and passing to the 
quotient group $G/T$ we can assume that $G$ is torsion-free and hence 
$G$ is free abelian.
\end{proof}

  
\begin{lemma}
Let $G$ be a group of finite rank which has a finite subgroup $D$
such that the quotient group $G/D$ is torsion-free abelian then $G$ has a
characteristic central torsion-free subgroup $A$ of finite index. 
\end{lemma}

\begin{proof}

As $G^{\prime }$ is a characteristic subgroup of $G$, it is not
difficult to show that $C={C_{G}}(G^{\prime })$ is a characteristic subgroup
of $G$. Then $B={C^{m}}$ is a characteristic subgroup of $G$, where $
m=\left\vert D\right\vert $. Since $\left\vert {G^{\prime }}\right\vert
<\infty $ and $C={C_{G}}(G^{\prime })$, we can conclude that $\left\vert {G:C
}\right\vert <\infty $. As the group $G$ has finite rank, it is easy to note
that $\left\vert {C:{C^{m}}}\right\vert <\infty $ and hence we can conclude
that $\left\vert {G:B}\right\vert <\infty $. For any $a,b\in C$ and any $
g\in G$ we have $[ab,g]={[a,g]^{b}}[b,g]$ and, as $[a,g]\in G^{\prime }$ and 
$b\in C={C_{G}}(G^{\prime })$, we can conclude that $[ab,g]=[a,g][b,g]$.\par 

Any element $a\in B={C^{m}}$ may be presented in the form $a={a_{1}}^{m}...{
a_{t}}^{m}$, where ${a_{i}}\in C$ and $1\leqslant i\leqslant t$. Since $
[ab,g]=[a,g][b,g]$, for any $g\in G$ we have $[a,g]={[{a_{1}},g]^{m}}...{[{
a_{t}},g]^{m}}$. Then, as $m=\left\vert D\right\vert $ and $[{a_{i}},g]\in
G^{\prime }\leqslant D$, where $1\leqslant i\leqslant t$, we can conclude
that $[a,g]=1$. It implies that $B$ is a characteristic central subgroup of
finite index in $G$. Since the quotient group $G/D$ is torsion-free, so is $
B/(B\cap D)$ and it implies that $A={B^{m}}$ is a torsion-free
characteristic central subgroup of finite index in $G$.
\end{proof}

\section{ Culling ideals of group rings of finitely generated nilpotent
groups}

A commutative ring $R$ is Hilbert if every prime ideal in $R$ is an
intersection of some maximal ideals of $R$. A ring $R$ is absolutely Hilbert
if it is commutative, Noetherian, Hilbert and all field images of $R$ are
locally finite. Let $G$ be a group and let $K$ be a normal subgroup of $G$.
Let $R$ be a ring and let $I$ be a $G$-invariant ideal of the group ring $RK$
then ${I^\dag } = G\bigcap {(I + 1)} $ is a normal subgroup of $G$.
We say that a $G$-invariant ideal $I$ is $G$-large if $R/(R\cap I ) = k$ is
a field and $|K / I^\dag| < \infty $. Let $R$ be an absolutely Hilbert ring,
let $V$ be an $RK$-module and $U$ be a submodule of $V$. According to \cite
[Section 1]{Broo88} a $G$-invariant ideal $I$ of $RK$ culls $U$ in $V$ if\par

(a) $VI < V$,\par

(b) $VI + U = V$,\par

(c) $R/(R\cap I ) = k$ is a field,\par

(d) $(I \cap RC)RK = I$, where $C = cch_p(K)$ is the centralizer of chief $p$
-factors of $K$ and $p = chark$.\par

However, (c) holds for any $G$-large ideal of $RK$ and (d) holds for any
nilpotent group because in this case $K = cch_p(K)$. So, in the case where
the group $K$ is nilpotent, a $G$-invariant ideal $I$ of $RK$ culls a
submodule $U$ of $V$ if $I$ satisfies conditions (a), (b) and (c), and if the 
ideal $I$ is $G$-large then $I$ culls  $U$ in $V$ if $I$
satisfies conditions (a) and (b).\par

An $RK$-module $V$ is $G$-ideal critical if for any non-zero submodule $U$
of $V$ there is a $G$-invariant ideal $I$ of $RK$ which culls the submodule $
U$ in $V$.\par

\begin{lemma}
Let $G$ be a finitely generated nilpotent group and let $K$ be a
normal subgroup of $G$. Let $R$ be an absolutely Hilbert ring. Let $V$ be an 
$RK$-module, let $U$ be a non-zero submodule of $V$ and let $J$ be a $G$
-invariant ideal of $RK$ which culls $U$ in $V$. Then there is a $G$-large
ideal $I$ of $RK$ which culls $U$ in $V$ and such that $J \leq I$.
\end{lemma}
\begin{proof}
It is sufficient to show that there is a $G$-invariant ideal $I$ of $
RK$ which culls $U$ in $V$ such that $J\leq I$, $R/(R\cap I)=k$ is a field
and $|K/I^{\dag }|<\infty $. This assertion was proved in the second
paragraph of section 5 of \cite{Broo88}.
\end{proof}

Let $\pi \subseteq\mathbb{N} $ be a set of prime integers, we will say that
a $G$-large ideal $I$ is $(\pi, G)$-large if $char(R/(R \cap I) \notin \pi $
. We will say that an $RK$-module $V$ is $(\pi, G)$-large critical if for
any non-zero submodule $U$ of $V$ there is a $(\pi, G)$-large ideal $I$ of $
RK$ which culls the submodule $U$ in $V$. Evidently, any $(\pi, G)$-large
critical $RK$-module is $G$-ideal critical.

\begin{lemma}
Let $G$ be a finitely generated nilpotent group and let $K$ be a
normal subgroup of $G$. Let $R$ be an absolutely Hilbert domain of
characteristic zero. Let $V$ be a $G$-ideal critical uniform $RK$-module
which is $R$-torsion-free. Then $V$ is $(\pi, G)$-large critical for any
finite set of prime integers $\pi \subseteq \mathbb{N}$.
\end{lemma}

\begin{proof}
Let $\pi = \{ {p_1},..., {p_k} \}$ be a finite set of prime integers
and let $U$ be a non-zero submodule of $V$. Put $X = ( \cap_{i = 1}^nV{p_i}
) \cap U$, as $char R = 0$ and $V$ is $R$-torsion-free, we see that $V{p_i}
\neq 0$ for all $1 \leq i \leq n$ and hence, as the module $V$ is uniform,
we can conclude that $X \neq 0$. Since $V$ is a $G$-ideal critical $RK$
-module, there is a $G$-invariant ideal $I$ of $RK$ which culls $X$ in $V$
and, by Lemma 4, we can assume that the ideal $I$ is $G$-large. By \cite[
Lemma 6]{Broo88}, the ideal $I$ culls $U$ and each $V{p_i}$, where, $1 \leq
i \leq n$. If $char(R/(R \cap I) = p_i$ for some $p_i \in \pi $ then $1p_i
\in I$, where $1$ is the unite of $R$. It implies that $V{p_i} \leq VI$ but
it is impossible because $I$ culls $V{p_i}$. Thus, $char(R/(R \cap I))
\notin \pi $.
\end{proof}

We should make some notes to apply Theorems E and D of \cite{Broo88} in the
next Proposition. A free abelian group $A$ of finite rank acted on by a
group $G$ is a $G$-plinth if all subgroups of finite index in $G$ act
rationally irreducibly on $A$. A polycyclic normal subgroup $K$ of a group $
G $ has a $G$-plinth series if it has $G$-invariant subgroups

\begin{equation}  \label{2}
1 = {K_1} \leq {K_2} \leq ... \leq {K_n} = K
\end{equation}

with each factor $K_{i+1}/K_{i}$ either finite or a $G$-plinth. It is easy
to note that if the group $G$ is nilpotent then any central series (\ref{2})
with cyclic factors is an $X$-plinth series for any subgroup $X$ of $G$.\par 

Let $R$ be an absolutely Hilbert ring. According to \cite[Section 2]{Broo88} 
an $RK$-module $V$ is $\mathcal{S}$-calibrated 
(with respect to the series (\ref{2})) if there are uniform 
$R{K_{i}}$-submodules such that\par 

(1) $V_1 \cong R/Q $ for some prime ideal $Q$ of $R$,\par 

(2) $V_i = V_1 RK_i$ for $1 \leq i \leq n $ and $V_n = V$,\par 

(3) either\par 

(i) $K_{i + 1}/K_i$ is infinite cyclic and $V_{i + 1} = V_i \otimes_{RK_i}
RK_{i + 1}$, or\par 

(ii) $u-dim_{RK_i} (V_{i + 1}) < \infty $ and $V_i$ is $K_{i + 1}$-essential
in $V_{i + 1}$, i.e. $V_i \cap Y \neq 0$ for any non-zero $RK_{i + 1}$
-submodule $Y \leq V_{i + 1}$.\par

Two $RK$-modules $V$ and $V_{1}$ are similar if their
injective hulls are isomorphic. If $V$ and $V_{1}$ are similar we denote $
V\approx V_1$. By \cite[Lemma 3.2(i)]{BBro85},  
$S_{G}(V)=\{g\in G|Vg\approx V\}$ is a subgroup of $G$ which is called the 
the similiser of the $RK$-module $V$ in $G$.\par 


\begin{proposition}
Let $K$ be a normal subgroup of a finitely generated
nilpotent group $G$. Let $R$ be an absolutely Hilbert domain of
characteristic zero. Let $W$ be a finitely generated $RK$-module
which is $R$-torsion-free. Then there is a cyclic uniform $RK$-submodule 
$aRK= V \leq W$ which is $(\pi, G)$-large critical for any finite set of
prime integers $\pi \subseteq \mathbb{N} $.
\end{proposition}

\begin{proof}
By \cite[Theorem 3.7]{Wehr09}, $W$ is a Noetherian $RK$-module and
hence, by \cite[Chap.6, Lemma 2.3]{McC}, $W$ has Krull dimension. Then, by 
\cite[Chap.6, Lemma 2.10]{McC}, $W$ has a critical submodule $U$ which, by 
\cite[Chap.6, Lemma 2.13]{McC}, is uniform. Let $S=S_{G}(U)$ then $
S=S_{S}(U) $. As the group $G$ is finitely generated nilpotent, it has a
central series with cyclic factors and hence the subgroup $K$ has a finite
series (\ref{2}) with cyclic factors which are centralized by the group $G$.
Then, as it was mentioned above, series (\ref{2}) may be considered as an $X$
-plinth series for any subgroup $X\leq G$. So, we can assume that (\ref{2})
is an $S$-plinth series. Therefore, by \cite[Theorem E]{Broo88}, $U$ has an 
$\mathcal{S}$-calibrated submodule $V$ with respect to the $S$-plinth series
( \ref{2}). However, (\ref{2}) is also a $G$-plinth series. Then, by 
\cite[Theorem D]{Broo88}, the module $V$ is $G$-ideal critical. 
By the definition of an $\mathcal{S}$-calibrated submodule,   $V=V_1 RK$, 
where $V_1 \cong R/Q$. It implies that $V_1 = aR$, where $a\in V_1$ and 
$Ann_R(a)=Q$, and hence $V=aRK$, i.e. the $RK$-module $V$ is cyclic.  
Then, as the module $V$ is $G$-ideal critical, the assertion follows 
from Lemma 5. 
\end{proof}

\section{ Commutative invariants for modules over finitely generated
nilpotent groups}

Let $P$ be a prime ideal of a commutative ring $S$ and let $S_{P}$ be the
localization of $S$ at the ideal $P$. Let $M$ be an $S$-module, the support $
Supp_{S}M$ of the module $M$ consists of all prime ideals $P$ of $S$ such
that $M_{P}=M\otimes_{S}S_{P}\neq 0$ (see. \cite[Chap. II,\S 4.4]{Bour}).\par

By \cite[Chap. IV, \S 1.4, Theorem 2]{Bour}, if $S$ and $M$ are Noetherian 
then the set $\mu_{S}(M)$ of minimal elements of $Supp_{S}M$ is finite and 
coincides with $\mu_{S}(Ann_{S}M)$, where $\mu_{S}(Ann_{S}M)$ is the 
set of prime ideals of $S $ which are minimal over $Ann_{S}M$, i.e. we have 

\begin{equation}  \label{3}
\mu_{S}(M)=\mu_{S}(Ann_{S}M)
\end{equation}

Further in the paper, if $J$ is an ideal of $S$ then $\mu_{S}(J)$ will 
denote the set of minimal prime ideals over $J$.\par

Let $G$ be a nilpotent group and let $K$ be a normal subgroup of $G$ such 
that the quotient group $G/K$ is finitely generated free abelian. Let $R$ 
be an absolutely Hilbert ring, if $I$ is a $G$-large ideal of the group 
ring $RK$ then the derived subgroup of the quotient group $G/I^{\dag }$ 
is finite and it follows from Lemma 3 that $G/I^{\dag }$ has a 
characteristic central subgroup $A$ of 
finite index. Since the quotient group $G/K$ is finitely generated free 
abelian, the subgroup $A$ may be chosen finitely generated free abelian.
Further in the paper, $A$ will denote a finitely generated free abelian central 
subgroup of $G/I^{\dag }$. Put $k=R/(R\cap I)$, by \cite[ Theorem 3.7]{Wehr09}, 
the group ring $kA$ is Noetherian and, by \cite[Chap.1, Lemma 1.1]{Pass89}, 
$kA$ is an integral domain. \par 

Let $W$ be a finitely generated $RG$-module then $W/WI$ is a finitely 
generated $kA$-module and hence $W/WI$ is a Noetherian 
$kA$-module. So, the finite set $\mu_{kA}(W/WI)=\mu_{kA}(Ann_{kA}(W/WI))$ 
is defined.\par

Further in the paper, for any $RG$-module $W$ we consider $W/WI$ as a 
$kA$-module, where $I$ is a $G$-large ideal of $RK$, $A$ is a
central free abelian subgroup of finite index in $G/I^{\dag }$ and  
$k=R/(R\cap I)$.\par

\begin{lemma}
 Let $G$ be a finitely generated nilpotent group and let $K$ be a
normal subgroup of $G$ such that the quotient group $G/K$ is free abelian.
Let $R$ be an absolutely Hilbert domain and let $W$ be a finitely generated 
$RG$-module. Let $I$ be a $G$-large ideal of $RK$, let $A$ be a central free 
abelian subgroup of finite index in $G/I^{\dag}$ and let $k=R/(R\cap I)$. 
Then for any $RG$-submodule $W_{1}\leq W$:\par

(i) $W_{1}/W_{1}I$ has a finite series each of whose quotient is isomorphic
to some section of $W/WI$;\par

(ii) $Supp_{kA}W_{1}/W_{1}I\subseteq Supp_{kA}W/WI$;\par

\end{lemma}

\begin{proof}

(i). By \cite[Chap.11, Theorem 2.9]{Pass89}, $WI^{n}\cap W_{1}\leq W_{1}I$
for some $n\in \mathbb{N}$. It implies that $W_{1}/W_{1}I$ has a finite
series each of whose factor is isomorphic to some section of $
WI^{k-1}/WI^{k} $, where $1\leq k\leq n$. Then it is sufficient to show that
each section $WI^{k-1}/WI^{k}$ has a finite series each of whose factor is
isomorphic to some section of $W/WI$. Applying induction on $k$, it is
sufficient to show that $WI/WI^{2}$ has a finite series each of whose factor
is isomorphic to some section of $W/WI$. As the ideal $I$ is $G$-invariant, $
J=IRG$ is an ideal of $RG$. Evidently, $WI^{k-1}/WI^{k}=WJ^{k-1}/WJ^{k}$ for
all $1\leq k\leq n$. By \cite[Chap.11, Theorem 3.12]{Pass89}, the ideal $J=IRG$
is polycentral, i.e. $J$ has generating elements $\{a_{1},a_{2},...,a_{r}\}$
such that $a_{i}+J_{i-1}$ is a central element of $RG/J_{i-1}$, where $
J_{i-1} $ is the ideal of $J$ generated by $\{a_{1},...,a_{i-1}\}$. As it
was mentioned above, it is sufficient to show that $WJ/WJ^{2}$ has a finite
series each of whose factor is isomorphic to some section of $W/WJ$. We
apply induction on the number of the generating elements 
$\{a_{1},a_{2},...,a_{r}\}$. Passing to the quotient module $W/WJ^{2}$ we can
assume that $WJ^{2}=0$. The mapping $\varphi:W\rightarrow W$ given by 
$\varphi :x\mapsto xa_{1}$ is an endomorphism such that $WJ\leq Ker\varphi$. 
Therefore, the submodule $\varphi (W)=Wa_{1}$ of $WJ$ is isomorphic to the 
section $W/Ker\varphi$ of $W/WJ$. Thus, passing to the quotient module 
$W/Wa_{1}$ we can assume that $Wa_{1}=0$. Then we can pass to the ideal 
$J/J_1$ of $RG/J_1$   and the assertion follows from the induction hypothesis.\par 

(ii) The assertion follows from (i) and \cite[Chap.II, \S 4, Proposition 16]
{Bour}.\par 

\end{proof}


\begin{proposition}
Let $G$ be a finitely generated nilpotent group and let $K$
be a normal subgroup of $G$ such that the quotient group $G/K$ is free
abelian. Let $R$ be an absolutely Hilbert domain, let $F$ be the field 
of fractions of $R$  and let $M$ be an irreducible $FG$-module. 
Let $I$ be a $G$-large ideal of $RK$, let $A$ 
be a central free abelian subgroup of finite index in $G/I^{\dag }$ and let 
$k=R/(R\cap I)$. Then for any non-zero elements $a,b\in M$:\par

(i) $aRG$ contains a submodule which is isomorphic to $bRG$;\par 

(ii)  $\mu_{kA}((akG)/(akG)I)=\mu_{kA}((bkG)/(bkG)I)$;\par

(iii) if $(akG)/(akG)I{\neq} 0$ then $(bkG)/(bkG)I{\neq} 0$; \par 
(iv) if $(akG)/(akG)I$ is $kA$-torsion then $(bkG)/(bkG)I$ is 
$kA$-torsion.

\end{proposition}

\begin{proof}
(i)  As the module $M$ is irreducible, there is $c\in FG$ 
such that $ac=b$. As $F$ is the field of fractions of the domain $R$, 
there is $d\in R$ such that $cd{\in} RG$. Therefore, $bd=acd\in aRG$. 
As $d\in R$ we see that $RG$-modules $bdRG$ and $bRG$ are isomorphic 
and the assertion follows.\par 

(ii) By (i), we can assume that  $bRG\leq aRG$. Then it follows from Lemma 6(ii) 
that  $Supp_{kA}(bRG)/(bRG)I\subseteq Supp_{kA}  (aRG)/(aRG)I$. However, 
the arguments of (i) show that we can assume that  $aRG\leq bRG$. Therefore, 
$Supp_{kA}(aRG)/(aRG)I\subseteq Supp_{kA} (bRG)/(bRG)I$ and hence 
$Supp_{kA}(aRG)/(aRG)I = Supp_{kA} (bRG)/(bRG)I$. It implies that 
$\mu_{kA}((akG)/(akG)I)=\mu_{kA}((bkG)/(bkG)I)$. \par

(iii) By (i), we can assume that  $aRG\leq bRG$. By \cite[Chap.11, Theorem 2.9]
{Pass89}, $(bRG){I}^{n}\cap (aRG)\leq (aRG)I$ for some $n\in \mathbb{N}$ and, 
as $(akG)/(akG)I{\neq} 0$, it implies that $(bRG){I}\neq (bRG)$, i.e. $(bkG)/(bkG)I$.\par

(iv) By (i), we can assume that  $bRG\leq aRG$. Then it follows from Lemma 6(i) 
that $(bRG)/(bRG)I$ has a finite series each of whose quotient is isomorphic
to some section of $(aRG)/(aRG)I$. Since the $kA$-module $(aRG)/(aRG)I$ is 
$kA$-torsion, it implies that $(bRG)/(bRG)I$ has a finite series each of whose 
quotient is $kA$-torsion. Then, as $kA$ is a domain, 
we can conclude that  $(bRG)/(bRG)I$ is $kA$-torsion. 
\end{proof}

  
\begin{proposition}
 Let $G$ be a finitely generated nilpotent group and let $
H\leq G$. Let $R$ be an absolutely Hilbert domain and let $0\neq W$ be an $
RG$-module. Let $K$ be a normal subgroup of $G$ such that $K\leq H$. Let $I$
be a $G$-large ideal of $RK$, let $k=R/(I\cap R),$ let $A$ be a free
abelian central subgroup of finite index in $G/I^{\dag }$ and ${B}=({H}/{\
I^{\dag }})\cap A$. Let $0\neq {d}\in W$ such that $({d}RG)/(dRG)I\neq 0$, 
let ${w}$ be the image of $d$ in $({d}RG)/({d}RG)I$ and let $J=Ann_{kA}(w)$. 
Suppose that ${d}RG={d}RH\otimes_{RH}RG$. Then:\par 

(i) $J=(J\cap kB)kA$;\par 

(ii) $\mu_{kA}(({d}RG)/(dRG)I)=\mu_{kA}(J)$;\par 

(iii) if $({d}RG)/(dRG)I\neq 0$ and $({d}RG)/(dRG)I$ is $kA$-torsion then 
$\mu_{kA}(J)$ consists of proper ideals of $kA$.
\end{proposition}

\begin{proof}

(i) Since ${d}RG={d}RH\otimes_{RH}RG$, by \cite[Chap. 2, Lemma 1.1]{Karp90}, 
we have ${d}RG=\oplus_{t\in T}({d}RH)t$, where $T$ is a right transversal 
to the subgroup $H$ in $G$. As $K\leq H$, we see that ${d}RGI=\oplus_{t\in T}({d}RHI)t$ 
and hence ${d}RG/{d}RGI=\oplus_{t\in T}({d}RH/{d}RHI)t$. It implies that ${w}RG=\oplus
_{t\in T}({w}RH)t$, where where ${w}$ is the image of $d$ in ${d}RG/{d}RGI$.
Then, as $I^{\dag }$ centralizes ${d}RG/{d}RGI$ and $R\cap I$ annihilates ${
d }RG/{d}RGI$, we have ${w}k(G/I^{\dag })=\oplus_ {t\in T} (wk(H/I^{\dag }))t$, 
where $T$ is a right transversal to $H/I^{\dag }$ in $G/I^{\dag }$.
Therefore, ${w}kA=\oplus_ {s\in S} (wkB)s$, where $S$ is a right transversal
to $B$ in $A$, and it implies that $Ann_{kA}(w)=J=(J\cap kB)kA$.\par

(ii) As  $I^{\dag }$ centralizes ${d}RG/{d}RGI$ and $R\cap I$ 
annihilates ${d}RG/{d}RGI$, we see that $({d}RG)/(dRG)I\cong{w}k(G/I^{\dag })$ 
and therefore $\mu_{kA}(({d}RG)/(dRG)I) = \mu_{kA}({w}k(G/I^{\dag }))$. 
Since $A$ is a central subgroup of finite index in $G/I^{\dag }$ it is not
difficult to show that $Ann_{kA}({w_{i}}k(G/I^{\dag }))= Ann_{kA}({w_{i}})$ 
and it implies that $\mu _{kA}(Ann_{kA}({w_{i}}k(G/I^{\dag })))= \mu
_{kA}(Ann_{kA}({w_{i}}))=\mu _{kA}(J)$. According to Equation (\ref{3}) 
we have $\mu _{kA}({w} k(G/I^{\dag }))=\mu _{kA}(Ann_{kA}({w_{i}}k(G/I^{\dag })))$  
and therefore, as $\mu_{kA}(({d}RG)/(dRG)I) = \mu_{kA}({w}k(G/I^{\dag }))$, we 
can conclude that $\mu_{kA}(({d}RG)/(dRG)I)=\mu_{kA}(J)$.\par

(iii) Put $U=({d}RG)/(dRG)I\neq 0$. According to (ii) and Equation (\ref{3}) 
it is sufficient to show that $\mu _{kA}(Ann_{kA}({U)})=\mu _{kA}(U)$ 
consists of proper ideals of $kA$. But, as the $kA$-module $U$ is finitely 
generated and $kA$ is a domain, $0\in \mu_{kA}(Ann_{kA}({U)})$ if and 
only if ${U}$ is not $kA$-torsion, and $kA\in \mu_{kA}(Ann_{kA}({U)})$ 
if and only if $U=0$. So, the assertion follows.
\end{proof}


\begin{lemma}
 Let $A$ be a finitely generated abelian group, let $k$ be a
field and let $I$ be an ideal of $kA$. Then for any subgroup $B$ of
finite index in $A$ the following assertions hold:\par

(i) for any ${Q\in \mu_{kB}}(I\cap kB)$ there is a prime ideal $P$ of $kA$
such that $I\leq P$, $Q=P\cap kB$ and any such prime $P$ is necessarily in ${
\mu_{kA}}(I)$; \par

(ii) if $P\in {\mu_{kA}}(I)$ then $P\cap kB\notin {\mu_{kB}}(I\cap kB)$ if
and only if there is $S\in {\mu_{kA}}(I)$ such that $S\cap kB<P\cap kB$;\par 

(iii) if $J$ is an ideal of $kA$ such that $\mu_{kA}(I)=\mu_{kA}(J)$ then $
\mu_{kB}(I\cap kB)=\mu_{kB}(J\cap kB)$;\par 

(iv) if $\mu_{kA}(I)$ consists of proper ideals of $kA$ then $\mu
_{kB}(I\cap kB)$ consists of proper ideals of $kB$.
\end{lemma}

\begin{proof}

(i) Evidently, the quotient ring $kA/I$ is integer over 
$(kB+I)/I\cong kB/(kB\cap I)$. 
Then it follows from \cite[Ch.V, \S 2, Theorem
3]{ZaSa58} that there is a prime ideal $P\leq kA$ such that $I\leq P$ and $
P\cap kB=Q$. If $P\notin {\mu_{kA}}(I)$ then there is an ideal $P^{\prime
}\in {\mu_{kA}}(I)$ such that $I\leq P^{\prime }<P$ and hence $I\cap kB\leq
P^{\prime }\cap kB\leq P\cap kB=Q$ and, as ${Q\in \mu_{kB}}(I\cap kB)$, we
can conclude that $I\cap kB\leq P^{\prime }\cap kB=P\cap kB=Q$. Since $I\leq
P^{\prime }<P$, it contradicts Complement 1 to \cite[Ch.V, \S 2, Theorem 3]
{ZaSa58}. Thus, $P\in {\mu_{kA}}(I)$.\par 

(ii) Let $P\in {\mu_{kA}}(I)$. Suppose that $P\cap kB\notin {\mu_{kB}}
(I\cap kB)$ then there is a prime ideal $Q\in {\mu_{kB}}(I\cap kB)$ such
that $Q<P\cap kB$. By (i), there is $S\in {\mu_{kA}}(I)$ such that $Q=S\cap
kB$ and hence $S\cap kB<P\cap kB$. Suppose now that there is $S\in {\mu
_{kA} }(I)$ such that $S\cap kB<P\cap kB$. Then $I\cap kB\leq S\cap kB<P\cap
kB$ and hence $P\cap kB\notin {\mu_{kB}}(I\cap kB)$.\par 

(iii) Let $Q\in \mu_{kB}(I\cap kB)$. By (i), there is $P\in {\mu_{kA}}(I)$
such that $Q=P\cap kB$. Then, by (ii), there are no $S\in {\mu_{kA}}(I)$
such that $S\cap kB<P\cap kB$. Since $\mu_{kA}(I)=\mu_{kA}(J)$, we can
conclude that there is $P\in {\mu_{kA}}(J)$ such that $Q=P\cap kB$ and
there are no $S\in {\mu_{kA}}(J)$ such that $S\cap kB<P\cap kB$. Then it
follows from (ii) that $Q\in \mu_{kB}(J\cap kB)$. The same arguments show
that if $Q\in \mu_{kB}(J\cap kB)$ then $Q\in \mu_{kB}(I\cap kB)$ and the
assertion follows.\par 

(iv) It follows from \cite[Lemma 2.2.3(ii)]{Tush2000} that if $P$ is a
proper ideal of $kA$ then $P\cap kB\neq 0$ and the assertion follows.
\end{proof}

 
\begin{lemma}
 Let $A$ be a free abelian group of finite rank and let $B$
be a subgroup of finite index in $A$. Let $k$ be a field such that $
chark\notin \pi (A/{B})$  and let $I$ be an ideal of $kA$ such that $
I=(I\cap kB)kA$. Suppose that there is a prime ideal $Q\in \mu_{kB}(I\cap kB)$
such that there is only one $P\in {\mu_{kA}}(I)$ such that $Q=kB\cap P$.
Then $P=QkA$.
 \end{lemma}

\begin{proof}

As $I=(I\cap kB)kA$ and $(I\cap kB)\leq Q$, we see that $QkA\geq I$.
If $P\in {\mu_{kA}}(I)$ and $Q=kB\cap P$ then $ P\geq QkA$ and
hence there is only one prime minimal over $QkA\geq I$. Therefore, it is
sufficient to show that the ideal $QkA$ is prime because then, by Lemma
7(i), $QkA\in {\mu_{kA}}(I)$. So, it is sufficient to show that if the
quotient ring $kA/QkA$ has only one minimal prime then $kA/QkA$ is a domain.
Since the quotient ring $kA/QkA$ is isomorphic to the cross product $(k{B}
/Q)\ast (A/{B})$ and $chark\notin \pi (A/{B})$, if follows from \cite[
Theorem 2.1]{Pass85} that the quotient ring $kA/QkA$ is semiprime and hence $
kA/QkA$ has no non-zero nilpotent elements. Therefore, if $kA/QkA$ is not a
domain then there are zero divisors $a,b\in kA/QkA$ such that $a\neq b$ and
hence there are at least two different minimal prime ideals of $kA$ over $
QkA $: minimal prime $P_{1}$ such that $a\notin P_{1}$ and minimal prime $
Q_{2} $ such that $b\notin Q_{2}$, but it leads to a contradiction.
\end{proof}

\section{ Primitive representations of finitely generated nilpotent groups}

Let $F$ be a field of and let $G$ be a finitely generated nilpotent group.
We say that a semi-primitive $FG$-module $M$ is purely semi-primitive if it
is not induced from a primitive $FX$-module for any subgroup $X\leq G$.


\begin{lemma}
Let $G$ be a finitely generated nilpotent group and let $H$ be a subgroup of $
G $. Let $F$ be a field and let $M$ be an $FG$-module such that ${M}={V}
\otimes _{FH}FG$, where ${V}$is an $FH$-module. Then:

(i) if $M$ is an irreducible $FG$-module then ${V}$ is an irreducible $FH$
-module;

(ii) if $M$ is a purely semi-primitive $FG$-module then ${V}$ is a purely
semi-primitive $FH$-module.

\end{lemma}

\begin{proof}

(i) It follows from \cite[Lemma 2.8(i, ii)]{Karp90} that if ${U}$ is a
proper submodule of ${V}$ then ${N}={U}\otimes _{FH}FG$ is a proper
submodule of ${M}$. As the $FG$-module  ${M}$ is irreducible, it implies
that the $FH$-module ${V}$ is irreducible.\par

(ii) If the module ${V}$ is not semi-primitive then there is a subgroup $K$
of $H$ such that ${V}={U}\otimes _{FK}FH$, where ${U}$ is an $FK$-submodule
of $V$ and $h(K)<h(H)$. Then, as ${M}={V}\otimes _{FH}FG$, it follows from 
\cite[Lemma 2.8(i. ii)]{Karp90} that ${M}={U}\otimes _{FK}FG$ but, as $
h(K)<h(H)\leq h(G)$, it is impossible because ${M}$ is semi-primitive.
Thus, ${V}$ is semi-primitive. \par

If ${V}$ is not purely semi-primitive then there is a subgroup $K$ of $H$
such that ${V}={U}\otimes _{FK}FH$, where ${U}$ is a primitive $FK$
-module. Then, as ${M}={V}\otimes _{FH}FG$, it follows from \cite[Lemma 2.1
]{Karp90} that ${M}={U}\otimes _{FK}FG$ but it is impossible because ${M}$
 is purely semi-primitive.

\end{proof}


\begin{lemma}
Let $F$ be a field and let $G$ be a finitely generated nilpotent group.
Suppose that there exists a purely semi-primitive irreducible $FG$-module $
M $. Then there are a descending chain $\{{G_{i}}|i\in \mathbb{N}\}$ of
subgroups of finite index in $G$ and a set $\{a_{i}|i\in \mathbb{N}
\}\subseteq M$ such that :

(i) $G/K$ is a free abelian group, where $K=\cap _{i\in \mathbb{N}}{G_{i}}$
and either for any prime integer $p$ there is $n\in \mathbb{N}$ such that $
p\notin \pi ({G_{n}}/{G_{i}})$ for all $i\geq n$ or $G/{G_{i}}$ is a $p$
-group for all $i\in \mathbb{N}$, where $p$ is a prime integer;

(ii) $a_{i}FG=a_{i}FG_{i}\otimes _{FG_{i}}FG$ for each $i\in \mathbb{N}$.

\end{lemma}

\begin{proof}

As the module $M$ is irreducible and purely semi-primitive, there are a
proper subgroup $H_{1}$ of finite index in $G$ and an $FH_{1}$-submodule 
$V_{1}$ of $M$ such that  $M=V_{1}\otimes _{FH_{1}}F{G}$. It follows from
Lemma 9 that the $FH_{1}$-submodule  $V_{1}$ is irreducible and purely
semi-primitive. Then the above arguments show that there are a proper
subgroup  $H_{2}$ of finite index in $H_{1}$ and an $FH_{2}$-submodule $
V_{2}$ of $V_{1}$ such that $V_{1}=V_{2}\otimes _{FH_{2}}FH_{1}$. As $
M=V_{1}\otimes _{FH_{1}}F{G}$ and $V_{1}=V_{2}\otimes _{FH_{2}}FH_{1}$, it
follows from \cite[Lemma 2.1]{Karp90} that $M=V_{2}\otimes _{FH_{2}}F{G}$.
It follows from Lemma 9 that the $FH_{2}$-submodule $V_{2}$ is irreducible
and purely semi-primitive. Continuing this process we obtain an infinite
strongly descending chain $\{{H_{i}}|i\in \mathbb{N}\}$ of subgroups of
finite index in $G$ and a strongly descending chain $\{V_{i}|i\in \mathbb{N}\}$, 
where $V_{i}$ is an irreducible and purely semi-primitive $F{H_{i}}${-module } 
such that $M=V_{i}\otimes _{FH_{i}}F{G}$ {for each }$i\in \mathbb{N}$. 
As $M$ is an irreducible $F{G}$-submodule and $V_{i}$ is an irreducible 
$F{H_{i}}$-module, we can conclude that for any  
$a_{i}\in V_{i}\backslash V_{i+1}$ we have $a_{i}FG=M$ and $a_{i}FH_{i}=V_{i}$ 
for each  $i\in \mathbb{N}$. Therefore, $a_{i}FG=a_{i}FH_{i}\otimes_{FH_{i}}F{G}$ 
for each $i\in \mathbb{N}$.\par 

Thus, we have an infinite strongly descending chain $\{{H_{i}}|i\in \mathbb{N
}\}$ of subgroups of finite index in $G$ and a set $\{a_{i}|i\in \mathbb{N}
\}\subseteq M$ such that $a_{i}FG=a_{i}FH_{i}\otimes _{FH_{i}}F{G}$ for
each $i\in \mathbb{N}$. As the chain $\{H_{i}|i\in \mathbb{N}\}$ is
strongly descending, we see that $|G:\cap _{i\in \mathbb{N}}{H_{i}}|=\infty $. 
 Then, by Proposition 1, there is a descending chain $\{{G_{i}}|i\in
\mathbb{N}\}$ of subgroups of finite index in $G$ such that ${H_{i}}\leq {
G_{i}}$ and the quotient group $G/K$ is free abelian, where $K=\cap _{i\in 
\mathbb{N}}{G_{i}}$.

(i) By Proposition 1, either for any prime integer $p$ there is $n\in 
\mathbb{N}$ such that $p\notin \pi ({G_{n}}/{G_{i}})$ for all $i\geq n$ or 
$G/{G_{i}}$ is a $p$-group for all $i\in \mathbb{N}$, where $p$ is a prime
integer.

(ii) As ${H_{i}}\leq {G_{i}}$ and $a_{i}FG=a_{i}FH_{i}\otimes _{FH_{i}}FG$, 
it follows from \cite[Lemma 2.1]{Karp90} that    ${a_{i}}FG={a_{i}}
FG_{i}\otimes_{FG_{i}}FG$ for all $i\in \mathbb{N}$.

\end{proof}


\begin{lemma}
Let $G$ be a finitely generated nilpotent group and let $K$ be a normal
subgroup of $G$ such that the quotient group $G/K$ is finitely generated
free abelian. Let $R$ be an absolutely Hilbert ring. Let $W$ be an $RG$
-module such that $W=VRG$ and $W\neq V\otimes_{RK}RG$, where $V$ is a
uniform $G$-ideal critical $RK$-module. Then there exists a submodule ${U}$
of $V$ with the property that when a $G$-large ideal $I$ of $RK$ culls ${U}$
in $V$ then $W/WI\neq 0$ and  $W/WI$ is ${kA}$
-torsion, where $k=R/(R\cap I)$ and $A$ is a free abelian central subgroup
of finite index in $G/I^{\dag }$. \end{lemma}

\begin{proof}

It follows from \cite[lemma 8]{Broo88} that there exists a non-zero
submodule $V_{1}$ of $V$ with the property that when a $G$-large ideal $I$
of $RK$ culls $V_{1}$ in $V$ then $W\neq V\otimes _{RK}RG$ if and only if $
W/WI\otimes _{kA}Q=0$, where $Q$ is the field of fractions of $kA$. As 
$W\neq V\otimes _{RK}RG$, if $I$ culls $V_{1}$ in $V$ then $W/WI\otimes
_{kA}Q=0$. However, it is well known that $W/WI\otimes _{kA}Q=0$ if and only
if $W/WI$ is $kA$-torsion.

It follows from \cite[lemma 14]{Broo88} that there is a non-zero submodule $
V_{2}$ of $V$ with the property that when a $G$-large ideal $I$ of $RK$
culls $V_{2}$ in $V$ then $W/WI\neq 0$. Let $U=V_{1}\cap V_{2}$ then, as $V$
is uniform, $U\neq 0$ and, by \cite[lemma 6]{Broo88}, if a $G$-invariant
ideal $I$ of $RK$ culls ${U}$ in $V$ then $I$ culls $V_{1}$ and $V_{2}$. So,
the assertion follows. \end{proof}


\begin{lemma}
Let $F$ be a finitely generated field of characteristic zero and let $G$ be a
finitely generated nilpotent group. Suppose that there exists a purely
semi-primitive irreducible $FG$-module $M$. Then there are an absolutely
Hilbert domain $R\leq F$, a descending chain $\{{G_{i}}|i\in \mathbb{N}\}$
of subgroups of finite index in $G$ and a set $\{a_{i}|i\in \mathbb{N}
\}\subseteq M$ such that the quotient group $G/K$ is free abelian, where $
K=\cap _{i\in \mathbb{N}}{G_{i}}$, and ${a_{i}}RG={a_{i}}RG_{i}\otimes
_{RG_{i}}RG$ for all $i\in \mathbb{N}$. Further, there is a $G$-large ideal $
I$ of $RK$ with a free abelian central subgroup ${A}$ of finite index in $
G/I^{\dag }$ and $k=R/(R\cap I)$ such that:

(i) $({a_{i}}RG)/({a_{i}}RG)I\neq 0$ and $({a_{i}}RG)/({a_{i}}RG)I$ is ${kA}$
-torsion for all $i\in \mathbb{N}$;

(ii) $\mu _{kA}(({a_{i}}RG)/({a_{i}}RG)I)=\mu _{kA}(({a_{j}}RG)/({a_{j}}RG)I)
$ for all $i,j\in \mathbb{N}$;

(iii) $chark\notin \pi (G_{1}/G_{i})$ for all $i\in \mathbb{N}$.

\end{lemma}
 \begin{proof}

As the field $F$ is finitely generated, there is a finitely generated domain 
$R$ such that $F$ is the field of fractions of $R$. It follows from \cite[
Chap.V,\S 3.4, Corollary 1]{Bour} that $R$ is an absolutely Hilbert domain.

Let $\{{G_{i}}|i\in \mathbb{N}\}$ be a descending chain of subgroups of
finite index in $G$ and $\{a_{i}|i\in \mathbb{N}\}\subseteq M$ be a set
which exist by Lemma 10. By Lemma 10(i), the quotient group $G/K$ is free
abelian, where $K=\cap _{i\in \mathbb{N}}{G_{i}}$. By Lemma 10(ii), $
a_{i}FG=a_{i}FG_{i}\otimes _{FG_{i}}FG$ and it implies that $RG={a_{i}}
RG_{i}\otimes _{RG_{i}}RG$ for all $i\in \mathbb{N}$. By Proposition 2 there
is a cyclic uniform $RK$-submodule $a_{0}RK=V\leq M$ which is $(\pi ,G)$
-large critical for any finite set of prime integers $\pi \subseteq \mathbb{N
}$.

 As the module $M$ is irreducible purely semi-primitive and $
h(K)<h(G)$, we see that $M={a_{0}}FG\neq {a_{0}}FK\otimes _{FK}FG$ and hence 
${a_{0}}RG\neq {a_{0}}RK\otimes _{RK}RG$. Then it follows from Lemma 11 that
there is a submodule $U\leq a_{0}RK$ with the property that when a $G$-large
ideal $I$ of $RK$ culls $U$ in $a_{0}RK$ then $({a_{0}}RG)/({a_{0}}RG)I\neq 0
$ and the $kA$-module $({a_{0}}RG)/({a_{0}}RG)I\neq 0$ is $kA$-torsion,
where $k=R/(R\cap I)$ and $A$ is a free abelian central subgroup of finite
index in $G/I^{\dag }$. As the $RK$-module $a_{0}RK$ is $(\pi ,G)$-large
critical, such an ideal $I$ exists. 

(i) As  $({a_{0}}RG)/({a_{0}}RG)I\neq 0$ and $({a_{0}}RG)/({a_{0}}RG)I$ 
is $kA$-torsion, it follows from Proposition 3(iii, iv) that 
$(a{_{i}}RG)/({a_{i}}RG)I\neq 0$ and $(a{_{i}}
RG)/({a_{i}}RG)I$ is $kA$-torsion for all $i\in \mathbb{N}$.

(ii) By Proposition 3(ii),  $\mu_{kA}(({a_{i}}RG)/({a_{i}}
RG)I)=\mu _{kA}(({a_{j}}RG)/({a_{j}}RG)I)$ for all $i,j\in \mathbb{N}$.

(iii) By Lemma 10(i), either for any prime integer $p$ there is $n\in 
\mathbb{N}$ such that $p\notin \pi ({G_{n}}/{G_{i}})$ for all $i\geq n$ or $
G/{G_{i}}$ is a $p$-group for all $i\in \mathbb{N}$, where $p$ is a prime
integer. Since the submodule ${a_{0}}RK$ is $(\pi ,G)$-large critical for
any finite set of prime integers $\pi \subseteq \mathbb{N}$, we can chose
the ideal $I$ such that $chark\neq p$ if $G/G_{i}$ is a $p$-group for all $
i\in \mathbb{N}$. Therefore, there is $n\in \mathbb{N}$ such that $
chark\notin \pi (G_{n}/G_{i})$ for all $i\geq n$. Starting the chain $\{{
G_{i}}|i\in \mathbb{N}\}$ and the set $\{{a_{i}}|i\in \mathbb{N}\}$ from the 
$n$-th members we obtain the assertion.
\end{proof}


\begin{lemma}

Let $F$ be a finitely generated field of characteristic zero and let $G$ be
a finitely generated nilpotent group. Suppose that there exists a purely
semi-primitive irreducible $FG$-module $M$. Then there exist a finitely
generated free abelian group $A$, a locally finite field $k$, a descending
chain $\{{A_{i}}|i\in \mathbb{N}\}$ of subgroups of finite index in $A$ and
a set $\{J_{i}|i\in \mathbb{N}\}$ of ideals of the group algebra $kA$ such
that:

(i) $\cap_{i\in \mathbb{N}}{A_{i}}=1$ and $chark\notin \pi (A_{1}/{A_{i}})$
for all $i\in \mathbb{N}$;

(ii) ${J_{i}}=({J_{i}}\cap k{A_{i}})kA$ for all $i\in \mathbb{N}$;

(iii) $\mu _{kA}({J_{i}})=\mu _{kA}({J_{j}})$ and $\mu _{kA}({J_{i}})$
consists of proper ideals of $kA$ for all $i,j\in \mathbb{N}$.

\end{lemma}

\begin{proof}

Let  $R\leq F$ be an absolutely Hilbert domain, $\{{G_{i}}|i\in \mathbb{N}\}$
be a descending chain of subgroups of finite index in $G$  and $\{{a_{i}}
|i\in \mathbb{N}\}$ be a subset of $M$  which exist by Lemma 12. Then the
quotient group $G/K$ is free abelian, where $K=\cap _{i\in \mathbb{N}}{G_{i}}
$, and ${a_{i}}RG={a_{i}}RG_{i}\otimes _{RG_{i}}RG$ for all $i\in \mathbb{N}$. 

Further, let  $I$ be a $G$-large ideal $I$ of $RK$ which exists by Lemma 12.
Let  $A$ be a free abelian central subgroup of finite index in $G/I^{\dag }$
and let $k=R/(R\cap I)$. Put $A_{i}=(G_{i}/I^{\dag })\cap A$ then $A_{i}$ is
a free abelian group for each $i\in \mathbb{N}$. As $\{{G_{i}}|i\in \mathbb{N
}\}$ is a descending chain of subgroups of finite index in $G$,   it is easy
to note that $\{{A_{i}}|i\in \mathbb{N}\}$ is a descending chain of
subgroups of finite index in $A$. Let ${w_{i}}$ be the image of ${a_{i}}$ in 
$({a_{i}}RG)/({a_{i}}RG)I$ and put ${J}_{i}=Ann_{kA}({w_{i})}$.

 (i) As $A_{i}$ is free abelian and $K/I^{\dag }$ is finite, we have 
$A_{i}\cap (K/I^{\dag })=1$. Then, as  $K/I^{\dag }=\cap _{i\in \mathbb{N}
}(G_{i}/I^{\dag })$, we can conclude that $\cap _{i\in \mathbb{N}}{A_{i}}=1$
. By Lemma 12(iii), $chark\notin \pi (G_{1}/G_{i})$ for all $i\in \mathbb{N}$
and hence $chark\notin \pi (A_{1}/{A_{i}})$ for all $i\in \mathbb{N}$.

(ii) As ${a_{i}}RG={a_{i}}RG_{i}\otimes _{RG_{i}}RG$, it follows from
Proposition 4(i) that ${J_{i}}=({J_{i}}\cap k{A_{i}})kA$ for all $i\in 
\mathbb{N}$.

(iii) By Lemma 12(ii), $\mu _{kA}(({d_{i}}RG)/({d_{i}}RG)I)=\mu _{kA}(({d_{j}
}RG)/({d_{j}}RG)I)$ for all $i,j\in \mathbb{N}$. Then it follows from
Proposition 4(ii) that $\mu _{kA}({J_{i}})=\mu _{kA}({J_{j}})$ for all $
i,j\in \mathbb{N}$. By Lemma 12(i), $({d_{i}}RG)/({d_{i}}RG)I\neq 0$ and $({
d_{i}}RG)/({d_{i}}RG)I$  is ${kA}$-torsion. Therefore, by Proposition
4(iii), $\mu _{kA}({J_{i}})$ consists of proper ideals of $kA$ for all $i\in 
\mathbb{N}$.  

\end{proof}


\begin{theorem}
Let $F$ be a finitely generated field of characteristic zero and let $G$ be
a finitely generated nilpotent group. Let $M$ be an irreducible $FG$-module.
Then there are an element $0\neq a\in M$ and a subgroup $L\leq G$ such that $
aFG=aFL\otimes _{FL}FG$ and $aFL$ is a primitive $FL$-module.

\end{theorem}

\begin{proof}

We apply induction on the Hirsch number $h(G)$ of $G$. If $h(G)=0$ then $G$
is finite. Therefore, $dim_{F}M$ is finite and the assertion follows. Then
by the induction hypothesis we can assume that the module $M$ is
semi-primitive. Suppose that the assertion does not hold then the module $M$
is purely semi-primitive.

By Lemma 13, there exist a finitely generated free abelian group $A$, a
locally finite field $k$, a descending chain $\{{A_{i}}|i\in \mathbb{N}\}$
of subgroups of finite index in $A$ and a set $\{J_{i}|i\in \mathbb{N}\}$ of
ideals of the group algebra $kA$ such that:

(i) $\cap _{i\in \mathbb{N}}{A_{i}}=1$ and $chark\notin \pi (A_{1}/{A_{i}})$
for all $i\in \mathbb{N}$;

(ii) ${J_{i}}=({J_{i}}\cap k{A_{i}})kA$ for all $i\in \mathbb{N}$;

(iii) $\mu _{kA}({J_{i}})=\mu _{kA}({J_{j}})$ and $\mu _{kA}({J_{i}})$
consists of proper ideals of $kA$ for all $i,j\in \mathbb{N}$.

As $\mu _{kA}({J_{i}})=\mu _{kA}({J_{j}})$ for all $i,j\in \mathbb{N }$,
we can put $\mu =\mu _{kA}({J_{i}})$ for all $i\in \mathbb{N}$.

It follows from (iii) and Lemma 7(iii) that $\mu _{kA_{t}}({J_{i}\cap }
kA_{t})=\mu _{kA_{t}}({J_{j}\cap }kA_{t})$ for all $i,j,t\in \mathbb{N}$.
Then we can put $\mu _{i}=\mu _{kA_{i}}({J_{j}\cap }kA_{i})$ for all $i,j\in 
\mathbb{N}$.

It follows from Lemma 7(i) that for any $J_{j}$ and any $Q\in \mu
_{kA_{i+1}}({J_{j}\cap kA}_{i+1})$ there is $P\in \mu _{kA_{i}}({J_{j}\cap kA
}_{i})$ such that $P\cap kA_{i+1}=Q$. It easily implies that $|\mu _{i}|\geq
|\mu _{i+1}|$ for all $i\in \mathbb{N}$. Therefore, there is $n\in \mathbb{N}
$ such that $|{\mu _{n}}|=|{\mu _{i}}|$ for all $i\geq n$.

Then for any $i\geq n$ and any $Q\in \mu _{kA_{i}}({J_{i}\cap kA}_{i})={\mu
_{i}}$ there is only one prime $P\in \mu _{kA_{n}}({J_{i}\cap kA}_{n})={\mu
_{n}}$ such that $P\cap kA_{i}=Q$. It follows from (i) that $chark\notin \pi
(A_{n}/{A_{i}})$ and it follows from (ii) that ${J_{i}}\cap k{A_{n}}=({J_{i}}
\cap k{A_{i}})kA_{n}$. Then it follows from Lemma 8 that $P=Q{kA}_{n}=(P\cap
kA_{i}){kA}_{n}$. On the other hand, as $|{\mu _{n}}|=|{\mu _{i}}|$, for
any $P\in \mu _{kA_{n}}({J_{i}\cap kA}_{n})={\mu _{n}}$ we have $P\cap
kA_{i}=Q\in \mu _{kA_{i}}({J_{i}\cap kA}_{i})={\mu _{i}}$. Therefore, for
any $i\geq n$ and any $P\in {\mu _{n}}$ we have $P=Q{kA}_{n}=(P\cap kA_{i}){
kA}_{n}$ and, as $\cap _{i\in \mathbb{N}}{A_{i}}=1$, it implies that $P={kA}
_{n}$. But, as $\mu $ consists of proper ideals of $kA$, it follows from
Lemma 7(iv) that ${\mu _{n}}$ consists of proper ideals of ${kA}_{n}$ and a
contradiction is obtained.

\end{proof}

\end{document}